\documentclass[]{article}

\pdfoutput=1

\usepackage{amsfonts}
\usepackage{amsmath}
\usepackage{amsthm}
\usepackage{amssymb}
\usepackage{graphicx}
\usepackage{epstopdf}
\usepackage{verbatim}
\usepackage{float}
\usepackage{tikz}
\usetikzlibrary{matrix,arrows,decorations.pathmorphing}
\usepackage{tikz-cd}
\usepackage[all,cmtip]{xy}

\theoremstyle{plain}
\newtheorem{theorem}{Theorem}[section]
\newtheorem{lemma}[theorem]{Lemma}

\newtheorem{prop}[theorem]{Proposition}

\theoremstyle{definition}

\DeclareMathOperator{\Span}{Span}

\DeclareMathOperator{\SU}{SU}

\DeclareMathOperator{\im}{Im}
\DeclareMathOperator{\Hess}{Hess}

\title{The multi-moment map of the nearly K\"ahler $S^3\times S^3$}
\author{Kael Dixon}

\begin{document}

\maketitle

\begin{abstract}
 We describe the multi-moment map associated to an almost Hermitian manifold which admits an action of a torus by holomorphic isometries. We investigate in particular the case of a $\mathbb T^3$ action on the homogeneous nearly K\"ahler $ S^3\times S^3$. We find that the multi-moment map in this case acts more-or-less similarly to the moment map of a toric manifold, while the more general case does not.
\end{abstract}
\section{Introduction}

An almost Hermitian manifold $(M,g,J)$ is \emph{nearly K\"ahler} if $\nabla^g J$ is skew-symmetric. We say a nearly K\"ahler manifold is \emph{strict} if it is not K\"ahler. The minimum dimension admitting strict nearly K\"ahler manifolds is $6$, and there are only a handful of known examples of compact strictly K\"ahler $6$-manifolds. The homogeneous spaces $\mathbb S^6, \mathbb S^3\times\mathbb S^3,\mathbb {CP}^3$ and $\SU_3/\mathbb T^2$ admit strict nearly K\"ahler structures, and there are no other homogeneous strict nearly K\"ahler $6$-manifolds \cite{butruille2010homogeneous}. The only non-homogeneous examples that are known are cohomogeneity one structures on $\mathbb S^6$ and $\mathbb S^3\times\mathbb S^3$ \cite{foscolo2017new}, and these are conjectured to be the only cohomogeneity one examples.

If one wants to look for higher cohomogeneity examples, one could look for strict nearly K\"ahler $6$-manifolds admitting the action of a 3-torus $\mathbb T^3$ by holomorphic isometries. Of the list of known examples in the previous paragraph, only the homogeneous $\mathbb S^3\times \mathbb S^3$ admits such a symmetry group. The purpose of this paper is to explore this example.

A compact K\"ahler $6$-manifold admitting a $\mathbb T^3$ action of holomorphic isometries would be toric. Such a manifold could be studied with use of the moment map $\mu$, which is a $\mathbb T^3$-equivariant map from the manifold to the dual Lie algebra of the torus, $\mathfrak t^*$. Each fiber of $\mu$ is a $\mathbb T^3$ orbit, and the image of $\mu$ is the polyhedron which is the convex hull of the $\mu$-image of the fixed points of the $\mathbb T^3$ action.

In general, an almost Hermitian $6$-manifold admitting a $\mathbb T^3$ action of holomorphic isometries would not be toric. However, we can study the multi-moment map $\nu$ associated to the $3$-form $d\omega$. This is a $\mathbb T^3$-equivariant map from the manifold to the three dimensional vector space $\Lambda^2\mathfrak t^*$, so one can hope that it will have similar properties to the momentum map of a toric $6$-manifold. We find that multi-moment map of $\mathbb S^3\times\mathbb S^3$ does have some similar properties and some differences with the momentum map of a toric $6$-manifold, while a more generic almost Hermitian manifold can have a rather poorly behaved multi-moment map.

We find that the multi-moment map image $\Delta:=\nu(\mathbb S^3\times\mathbb S^3)$ of $\mathbb S^3\times \mathbb S^3$ is convex and that its boundary $\partial\Delta$ contains the $1$-skeleton of a regular tetrahedron. However, $\Delta$ bulges beyond the faces of the tetrahedron, and $\partial\Delta$ is smooth away from the vertices. Along $\partial\Delta$, each $\nu$-fiber is a $\mathbb T^3$ orbit, but in the interior, each fiber contains two orbits. The following table compares the fiber types for the multi-moment map of $\mathbb S^3\times\mathbb S^3$ to the moment map of a toric $3$-manifold:\newline

\begin{tabular}{c|c|c}
	Fiber of a point in ... & $\mu$ toric $6$-manifold & $\nu$ for nearly K\"ahler $S^3\times\mathbb S^3$
	\\\hline
	a vertex & \{point\} & $\mathbb T^2$ \\
	an edge	& $\mathbb S^1$ & $\mathbb T^3$ \\
	a face & $\mathbb T^2$ & $\mathbb T^3$ \\
	the interior & $\mathbb T^3$ & $\mathbb T^3 \coprod\mathbb T^3$ 
\end{tabular}

\section{Torus actions on almost Hermitian structures}
Let $(M,g,J,\omega)$ be an almost Hermitian manifold. Let $\mathbb T$ be a torus acting on $M$ by holomorphic isometries. Any vector $X\in\mathfrak t$ induces a vector field $K_X$ on $M$, which is a holomorphic Killing vector field. This means that $\mathcal L_{K_X}g=0=\mathcal L_{K_X}J$. By the Leibniz rule, this implies that $\mathcal L_{K_X}\omega=0$.

If $(M,g,J,\omega)$ is K\"ahler, so that $\omega$ is closed, then there exists a moment map $\mu:M\to\mathfrak t^*$ defined by
$$\langle d\mu,X\rangle=-K_X\lrcorner\,\omega,$$
where $\langle\cdot,\cdot\rangle$ is the natural pairing of $\mathfrak t$ and $\mathfrak t^*$.

If we do not require $(M,g,J,\omega)$ to be K\"ahler, there is a multi-moment map associated to the closed $3$-form $d\omega$ \cite{madsen2013closed}. This is the map $\nu:M\to \Lambda^2\mathfrak t^*$ defined by 
$$\left\langle d\nu,\sum_i X_i\wedge Y_i\right\rangle = -\sum_i K_{X_i}\lrcorner \left(K_{Y_i}\lrcorner\, d\omega\right),\quad \forall \sum_{i}X_i\wedge Y_i\in\Lambda^2\mathfrak t,$$
where here $\langle\cdot,\cdot\rangle$ is the natural pairing of $\Lambda^2\mathfrak t$ and $\Lambda^2\mathfrak t^*$. Recall that the Lie derivative acts on differential forms by $$\mathcal L_X\tau= d(X\lrcorner\,\tau)+X\lrcorner\, d\tau.$$
We can use this to simplify our expression for the multimoment map $\nu$:
\begin{align*}
	\left\langle d\nu,\sum_i X_i\wedge Y_i\right\rangle 
	=& -\sum_i K_{X_i}\lrcorner 
	\left(K_{Y_i}\lrcorner\, d\omega\right)
	= -\sum_i K_{X_i}\lrcorner \left(\mathcal L_{K_{Y_i}}\omega-d\left(K_{Y_i}\lrcorner\,\omega\right)\right)\\
	=& \sum_i K_{X_i}\lrcorner\, d\left(K_{Y_i}\lrcorner\,\omega\right)
	= \sum_i \mathcal L_{K_{X_i}}\left(K_{Y_i}\lrcorner\,\omega\right)-d\left(K_{X_i}\lrcorner\, K_{Y_i}\lrcorner\,\omega\right) \\
	=& \sum_i d\omega(K_{X_i},K_{Y_i}).
\end{align*}
Here we've used the fact that $\mathcal L_{K_X}\omega=0=[K_X,K_Y]$ and the Leibniz rule to get $\mathcal L_{K_{X_i}}\left(K_{Y_i}\lrcorner\,\omega\right)=0$.  This equation can be integrated to solve for $\nu$: 
\begin{align*}\label{eqnMultoMomHerm}
\nu\left(\sum_i X_i\wedge Y_i\right)
= \sum_i \omega(K_{X_i},K_{Y_i}) + C
\end{align*}
for some constant $C$. Note that we can always choose $C$ to be $0$, so we will.

Note that one cannot expect $\nu$ to behave well for an arbitrary Hermitian structure. Motivated by the behaviour of the moment map of toric manifolds, one could expect that $\nu$ is almost everywhere a submersion, which means that the (multi-)moment map locally separates orbits. The following proposition shows that this condition does not always hold:
\begin{prop}
	Let $(M,g,J)$ be an almost Hermitian manifold equipped with a torus $\mathbb T$ acting by holomorphic isometries. Then there exists a metric $\hat g$ related to $g$ by a $\mathbb T$-invariant conformal factor such that the multimoment map $\hat\nu$ of $(M,\hat g, J, \mathbb T)$ is not a submersion on some open set in $M$.
\end{prop}
\begin{proof}	
	If $\nu(M)=\{0\}$, then $\hat g = g$ satisfies the claimed property. Otherwise, there exists some $p_0\in M$ with $\nu(p_0)\neq 0$.
	We can choose a smooth $\mathbb T$-invariant function $\phi$ so that $\phi(p)=-\log\|\nu(p)\|$ for all $p$ in some neighbourhood $U$ of $p_0$. 
	Consider the conformally related metric $\hat g = e^\phi g$. The multi-moment map with respect to the conformally related K\"ahler form $\hat \omega=e^\phi\omega$ is $\hat\nu= e^\phi \nu$.  We chose $\phi$ so that $\hat\nu$ maps $U$ into the unit sphere in $\Lambda^2\mathfrak t^*$, so that $\hat \nu$ is not a submersion on $U$.
\end{proof}

In the rest of the paper, we will describe the multi-moment map for a torus action on the homogeneous nearly K\"ahler $\mathbb S^3\times \mathbb S^3$. We find that $\nu$ is a submersion near generic orbits, and show other similarities and differences to the moment map of a toric manifold.

\section{Homogenous nearly K\"ahler $\mathbb S^3\times\mathbb S^3$}

We begin by reviewing the definition of the homogenous nearly K\"ahler structure on $\mathbb S^3\times\mathbb S^3$, following the work in \cite{dioos2016lagrangian}.

We identify $\mathbb S^3$ with the unit sphere in the quaternions $\mathbb H$. For any $p\in\mathbb S^3$, $T_p\mathbb S^3\subset T_p\mathbb H$ is the image of $T_1\mathbb S^3$ by the pushforward of left-multiplication by $p$. Identifying $T_p\mathbb S^3\subset T_p\mathbb H$ with $p^\perp\subset\mathbb H$, this pushforward is simply quaternionic multiplication by $p$.
Thus the basis $\{i,j,-k\}$ of $\im\mathbb H$ which is identified with $T_1\mathbb S^3$ gives a frame for $T_{(p,q)}\mathbb S^3\times\mathbb S^3=T_p\mathbb S^3\oplus T_q\mathbb S^3:$
\begin{align*}
 E_1(p,q) = (pi,0), \qquad & F_1(p,q) = (0,qi),\\
 E_2(p,q) = (pj,0), \qquad & F_2(p,q) = (0,qj),\\
 E_3(p,q) = (-pk,0), \qquad & F_3(p,q) = (0,-qk),
\end{align*}
where $i,j,k$ are imaginary quaternions satisfying $ij=k$.

The almost complex structure for the homogenous nearly K\"ahler $\mathbb S^3\times \mathbb S^3$ is given in this frame by
$$J = \frac1{\sqrt 3}\sum_{n=1}^3\left(-E_n\otimes E^n+F_n\otimes F^n + 2 F_n\otimes E^n - 2 E_n\otimes F^n\right).$$
The metric $g$ is given by the average of $g_{\mathbb H^2}$ and $g_{\mathbb H^2}(J\cdot, J\cdot)$, where $$g_{\mathbb H^2}=\sum_{n=1}^3\left(\left(E^n\right)^2+\left(F^n\right)^2\right)$$ is the flat metric from $\mathbb H^2$ restricted to $\mathbb S^3\times\mathbb S^3$. This gives
\begin{align*}
g=& \frac43\sum_{n=1}^3\left((E^n)^2-E^nF^n+(F^n)^2\right),\\
\omega =& \frac4{\sqrt 3}\sum_{n=1}^3 E^n\wedge F^n.
\end{align*}

\subsection{Torus actions}

For unit quaternions $a,b,c\in\mathbb S^3$, the map $$F_{a,b,c}:\mathbb S^3\times\mathbb S^3\to\mathbb S^3\times \mathbb S^3:(p,q)\mapsto(apc^{-1},bqc^{-1})$$ is a holomorphic isometry \cite{podesta2010six}.

\begin{lemma}
	The map $$F:\mathbb S^3\times\mathbb S^3\times\mathbb S^3\to {Aut}(\mathbb S^3\times\mathbb S^3,J)\cap{Isom}(\mathbb S^3\times\mathbb S^3,g):(a,b,c)\mapsto F_{a,b,c}$$ is an injective homomorphism.
\end{lemma}
\begin{proof}
	It is clear from the definition of $F$ that $F_{a,b,c}\circ F_{a',b',c'}= F_{aa',bb',cc'}$, so that $F$ is a homomorphism. To see that $F$ is injective, let $F_{a,b,c}=Id$. Then $$(1,1)=F_{a,b,c}(1,1)=(ac^{-1},ab^{-1}),$$ so that $a=b=c$.
	For any $(p,q)\in\mathbb S^3\times\mathbb S^3$, $$(p,q)=F_{a,a,a}(p,q)=(apa^{-1},aqa^{-1}).$$
	Since this is true for all $p,q\in\mathbb S^3$, we find that $a$ lies in the center of $\mathbb S^3$. Since $\mathbb S^3$ has a trivial center, $a=1$ as required.
\end{proof}

Since the projection of $\mathbb S^3\times\mathbb S^3\times\mathbb S^3$ onto any of its factors is a homomorphism, any abelian subgroup of $\mathbb S^3\times\mathbb S^3\times\mathbb S^3$ must be a product of abelian subgroups of $\mathbb S^3$. But every non-trivial abelian subgroup of $\mathbb S^3$ is of the form $\left\{e^{At}\right\}_{t\in\mathbb R}$ for some unit imaginary quaternion $A$. Thus, a maximal torus in $\mathbb S^3\times\mathbb S^3\times\mathbb S^3$ is of the form 
$$\{(e^{At_1},e^{Bt_2},e^{Ct_3})\}_{(t_1,t_2,t_3)\in\mathbb R^3},$$
for some $A,B,C\in\mathbb S^2$, identifying $\mathbb S^2$ with the unit imaginary quaternions. A routine computation shows that the image of such a torus under $F$ is generated by the Killing vector fields
\begin{align*}
	K_1 =& (Ap,0),\\
	K_2 =& (0,Bq), \\
	K_3 =& (-pC,-qC).
\end{align*}
Since $p\in\mathbb S^3$, $Ap=p\bar pAp$, so that these Killing vector fields can be written in terms of the frame $(E_1,E_2,E_3,F_1,F_2,F_3)$ as 
\begin{align*}
	K_1 =& \big((\bar pAp)\cdot i\big) E_1 + \big((\bar pAp)\cdot j\big) E_2 - \big((\bar pAp)\cdot k\big) E_3, \\
	K_2 =& \big((\bar qBq)\cdot i\big) F_1 + \big((\bar qBq)\cdot j\big) F_2 - \big((\bar qBq)\cdot k\big) F_3, \\
	K_3 =& (C\cdot i) (E_1+F_1) + (C\cdot j) (E_2+F_2) - (C\cdot k) (E_3+F_3),
\end{align*}
where $\cdot$ is the dot product on $\mathbb H$. This allows us to compute
\begin{align*}
	\frac{\sqrt 3}4\omega(K_1,K_2) =& (\bar pAp)\cdot(\bar qBq), \\
	\frac{\sqrt 3}4\omega(K_1,K_3) =& (\bar pAp)\cdot C, \\
	\frac{\sqrt 3}4\omega(K_2,K_3) =& (\bar qBq)\cdot C.
\end{align*}
Choosing the basis $\left\{\frac{\sqrt 3}4(K_n\wedge K_m)^*\right\}_{1\leq n<m\leq 3}$ for $\Lambda^2\mathfrak t^*$, this allows us to write the multi-moment map as
$$\nu(p,q)=\big((\bar pAp)\cdot(\bar qBq),(\bar pAp)\cdot C,(\bar qBq)\cdot C\big).$$

Note that $\nu^{-1}(0)$ is the union of the Lagrangian torus orbits. The example of a Lagrangian torus in \cite{dioos2016lagrangian} can be found with the values $A=B=i$ and $C=j$.

\subsection{Behaviour of the multi-moment map}

We will first describe the image of the multi-moment map $\nu$. Then we will describe the structure of its fibers.

For $X\in\mathbb S^2$, let us define a map
$$\pi_X:\mathbb S^3\to\mathbb S^2:p\mapsto \bar p X p.$$
When $X=i$, this is the usual Hopf fibration. For general $X$, $\pi_X$ also identifies $\mathbb S^3$ as a $\mathbb S^1$ bundle over $\mathbb S^2$.

Define a function 
$$\bar\nu:\mathbb S^2\times\mathbb S^2\to\Lambda^2\mathfrak t^*:(x,y)\mapsto (x\cdot y, x\cdot C, y\cdot C),$$
so that $\nu=\bar{\nu}\circ(\pi_A\times\pi_B).$ Let $\Delta=\nu(\mathbb S^3\times\mathbb S^3)=\bar \nu(\mathbb S^2\times \mathbb S^2)$ with interior $\mathring \Delta$.

\begin{lemma}\label{lemDeltaDiffGraphs}
	$\Delta=\bigg\{(X,Y,Z):X\in\big[f_-(Y,Z),f_+(Y,Z)\big]\bigg\}$, where $$f_\pm(Y,Z)=YZ\pm\sqrt{1-Y^2}\sqrt{1-Z^2}.$$
\end{lemma}
\begin{proof}
	Let $C^\perp\leq \im\mathbb H$ be the plane orthogonal to $C$. Use the orthogonal decomposition $\im\mathbb H = \mathbb R C \oplus C^\perp$ to write any $(x,y)\in\mathbb S^2\times\mathbb S^2$ as
	$$x=(x\cdot C) C + x^\perp, \quad y=(y\cdot C) C + y^\perp, \quad x^\perp,y^\perp\in C^\perp.$$
	Then we have the following relations:
	\begin{align*}
	 1 &= x\cdot x = (x\cdot C)^2 + \|x^\perp\|^2, \\
	 1 &= y\cdot y = (y\cdot C)^2 + \|y^\perp\|^2, \\
	 x\cdot y &= (x\cdot C)(y\cdot C) + x^\perp\cdot y^\perp.
	\end{align*}
	If $\bar \nu (x,y) = (X,Y,Z)$, then 
	$$X = x\cdot y = YZ +x^\perp\cdot y^\perp.$$
	By the Cauchy-Schwarz inequality, $|x^\perp\cdot y^\perp|\leq \|x^\perp\|\|y^\perp\|=\sqrt{1-X^2}\sqrt{1-Y^2},$ so that $f_-(Y,Z)\leq X \leq f_+(Y,Z)$. It is clear that by varying $x$ and $y$, any value of $X$ in this range can be attained, proving the claimed result.
\end{proof}

\begin{lemma}
	$\Delta$ is convex.
\end{lemma}
\begin{proof}
	By the previous lemma, it suffices to prove that $\mp f_\pm$ is a convex function. This follows from the computation $$\det\circ\Hess(f_\pm)=\left(\frac{X\sqrt{1-Y^2}+Y\sqrt{1-X^2}}{\sqrt{1-X^2}\sqrt{1-Y^2}}\right)^2\geq 0.$$
\end{proof}

\begin{prop}
	$\partial\Delta$ is contained in the affine variety $0=F(X,Y,Z)=2XYZ-X^2-Y^2-Z^2+1$. The set of singular points of $\partial\Delta$ is
	$$V:=\big\{(1,1,1),(1,-1,-1),(-1,1,-1),(-1,-1,1)\big\}.$$
\end{prop}
\begin{proof}
	By lemma \ref{lemDeltaDiffGraphs}, $\partial\Delta$ is given by points $(X,Y,Z)\in\mathbb R^3$ with $X=f_\pm(Y,Z)$. Such a points satisfy the relation $X^2 = 2XYZ+(1-Y^2)(1-Z^2)-Y^2Z^2$, which can be rearranged to form $F(X,Y,Z)=0$.
	
	The singular points of $\partial\Delta$ are the points where $F$ and $\nabla F$ both vanish. The set of points where $\nabla F$ vanishes are $\{(0,0,0)\}\cup V$. The result follows since $F(0,0,0)=1\neq 0$, while $F$ vanishes on $V$.
\end{proof}

\begin{prop}
	The line segment between any two points in $V$ lies in $\partial\Delta$.
\end{prop}
\begin{proof}
	We will show that the line segment between $(1,1,1)$ and $(1,-1,-1)$ lies in $\partial\Delta$, with the other line segments following similarly. This line segment is parametrized by $$\gamma:[-1,1]\to\mathbb R:t\mapsto (1,t,t).$$ 
	Consider the functions $\tilde f_\pm(X,Y,Z):=f_\pm(Y,Z)-X$.
	By lemma \ref{lemDeltaDiffGraphs}, 
	$$\partial\Delta=\left\{\vec X\in\mathbb R: \tilde f_+(\vec X)=0\geq\tilde f_-(\vec X) \text{ or } \tilde f_+(\vec X)\geq 0=\tilde f_-(\vec X)\right\}.$$
	We compute
	$$\tilde f_\pm\circ\gamma(t)= t^2-1\pm(1-t^2).$$
	Thus for $t\in [-1,1]$, $\tilde f_+\circ\gamma(t)=0$ and $\tilde f_-\circ\gamma(t)=2(t^2-1)\leq 0$. This shows that $\gamma(t)\in\partial\Delta$ as required.
\end{proof}

By the previous proposition, we find that $\partial\Delta$ contains the $1$-skeleton of the regular tetrahedron with vertices $V$. However, the full tetrahedron is properly contained in $\Delta$. In Figure 1, we see that $\Delta$ is a regular tetrahedron with convexly bulging sides:

\begin{figure}
	\centering
	\includegraphics[scale=0.7]{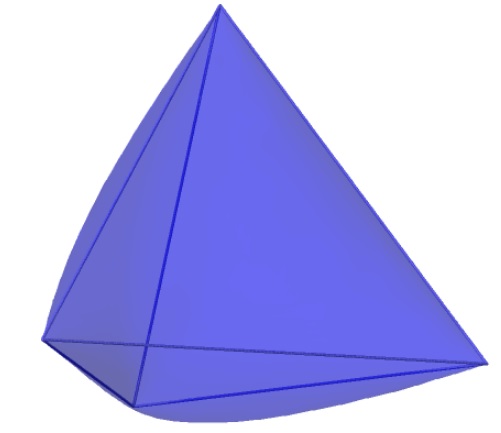}
	\caption{The multi-moment map image of $\mathbb S^3\times\mathbb S^3$}
\end{figure}

\begin{prop}
	$\bar \nu$ has three different orbit types according to the following table:
	\begin{center}\begin{tabular}{l|l|l|l}
			$\dim\Span\{x,y,C\}$ & location on $\Delta$ & $\bar \nu^{-1}(\bar \nu(x,y))$ & $\nu^{-1}(\bar \nu(x,y))$\\\hline
			1 & $V$	& $\{(x,y)\}$ & $\mathbb T^2$ \\
			2 & $\partial\Delta\backslash V$ & $\mathbb S^1$ 
			  & $\mathbb T^3$\\
			3 & $\mathring\Delta$	
			  & $\mathbb S^1 \coprod \mathbb S^1$ 
			  & $\mathbb T^3 \coprod \mathbb T^3$
	\end{tabular}\end{center}
\end{prop}
\begin{proof}
	Let $(x,y)\in\mathbb S^2\times\mathbb S^2$ such that $\dim\Span\{x,y,C\}\neq 1$. Then one of $x$ or $y$ is not $\pm C$. We will treat the case when $x\notin\{\pm C\}$, with the other case following similarly. 
	
	Write $\bar \nu(x,y) =\tau=(\tau_1,\tau_2,\tau_3)$. Let $(x_0,y_0)\in\bar \nu^{-1}(\bar \nu(x,y))$. Thus $\tau_2=x_0\cdot C\notin\{\pm 1\}.$ This relation defines a circle $S_0$ on $\mathbb S^2$ of possible $x$ values. For a fixed $x_0\in S_0$, the relations $\tau_1=x_0\cdot y_0$ and $\tau_3=y_0\cdot C$ define two circles $S_1$ and $S_2$ on $\mathbb S^2$ centered at $x_0$ and $C$ respectively, which intersect at possible solutions for $y_0$. Two circles can intersect in at most $2$ points. If $S_1$ and $S_2$ do not intersect, then $\bar \nu^{-1}(\tau)=\emptyset$, contradicting $\tau\in\Delta$. If $S_1$ and $S_2$ intersect at exactly one point $y_0$, then $y_0$ is a linear combination of the centers $x_0$ and $C$ of $S_1$ and $S_2$. If they intersect at two points, then each intersection point is not a linear combination of $x_0$ and $C$. Since there is a circle worth of choices for $x_0$, this gives the last two rows of the table.
	
	The remaining points in $\mathbb S^2\times\mathbb S^2$ satisfy $\dim\Span\{x,y,C\}=1$. This is equivalent to $x,y\in\{\pm C\}$, which define $4$ points. To see that these points live in different $\bar \nu$ fibres, the following table evaluates $\bar \nu$ at each of these points:
	\begin{center}\begin{tabular}{l|l|l}
			$x$ & $y$ & $\bar \nu(x,y)$ \\
			\hline
			$C$ & $C$ & (1,1,1) \\
			$C$ & $-C$ & (-1,1,-1) \\
			$-C$ & $C$ & (-1,-1,1) \\
			$-C$ & $-C$ & (1,-1,-1)
	\end{tabular}\end{center}
	Thus the singleton fibres get mapped to $V$. We've established the correspondence between the first and third rows in the claimed table. The last column follows since $\nu=\bar \nu\circ (\pi_A\times\pi_B)$, where $\pi_A\times\pi_B$ determined a $\mathbb T^2$ bundle. The second column follows from the description in lemma \ref{lemDeltaDiffGraphs}, noting that  $\partial \Delta$ consists of the points where the Cauchy-Schwarz inequality is an equality, which are the points where $\{x,y,C\}$ are linearly dependent vectors in $\im\mathbb H$.
\end{proof}
Note that $\bar \nu^{-1}\big(\mathring\Delta\big)$ has two connected components determined by the sign of $\det\{x,y,C\}$, while $\bar \nu^{-1}(\partial \Delta)$ is the vanishing locus of $\det\{x,y,C\}$. It follows that $\nu$ is a submersion along $\nu^{-1}\big(\mathring\Delta\big)$.

\section{Acknowledgments}

I'd like to thank Dr. Uwe Semmelmann for pointing me in the direction of this project. This work was done while funded by the Belgian Science Policy under the Interuniversity Attraction Pole \emph{Dynamics, Geometry, and Statistical Physics}.

\bibliography{myBib}{}
\bibliographystyle{plain}
\end{document}